\documentclass[11pt]{amsart} \usepackage{latexsym, amssymb, stmaryrd}
\usepackage[T1]{fontenc}

\DeclareTextSymbol{\thh}{T1}{254}

\newtheorem{thm}{Theorem}[section]

\newtheorem{prop}[thm]{Proposition}
\newtheorem{cor}[thm]{Corollary}

\theoremstyle{definition}
\newtheorem{df}[thm]{Definition}

\newtheorem{facts}[thm]{Facts}

\newcommand{\R}{\mathbb{R}}
\newcommand{\Z}{\mathbb{Z}}

\newcommand{\curly}[1]{\mathcal{#1}}

\newcommand{\la}{\curly{L}}

\makeatletter

\def\indsym#1#2{%
  \setbox0=\hbox{$\m@th#1x$}%
  \kern\wd0%
  \hbox to 0pt{\hss$\m@th#1\mid$\hbox to 0pt{$\m@th#1^{#2}$}\hss}%
  \lower.9\ht0\hbox to 0pt{\hss$\m@th#1\smile$\hss}%
  \kern\wd0}

\def\nindsym#1#2{%
  \setbox0=\hbox{$\m@th#1x$}%
  \kern\wd0%
  \hbox to 0pt{\hss$\m@th#1\not$\kern1.4\wd0\hss}
  \hbox to 0pt{\hss$\m@th#1\mid$\hbox to 0pt{$\m@th#1^{\,#2}$}\hss}%
  \lower.9\ht0\hbox to 0pt{\hss$\m@th#1\smile$\hss}%
  \kern\wd0}

\def\dotminussym#1#2{%
  \setbox0=\hbox{$\m@th#1-$}%
  \kern.5\wd0%
  \hbox to 0pt{\hss\hbox{$\m@th#1-$}\hss}%
  \raise.6\ht0\hbox to 0pt{\hss$\m@th#1.$\hss}%
  \kern.5\wd0}

\def \r { {\mathbb R} }
\def \<{\langle}
\def \>{\rangle}
\def \n {\mathbb N}

\def \*Z {{{^*}\Z}}

\def \((  {(\!(}
\def \)) {)\!)}

\def \int{\operatorname{int}}
\numberwithin{equation}{section}

\def \k{\mathcal{K}}

\def \u{\mathcal{U}}

\def\R{\mathcal R}
\def \Th{\operatorname{Th}}

\allowdisplaybreaks[2]

\begin{document}

\title{Amalgamating $\R^\omega$-embeddable von Neumann algebras}
\author{Ilijas Farah, Isaac Goldbring, Bradd Hart}
\thanks{The authors would like to thank NSERC and the Fields Institute for supporting this work.  Goldbring's work was partially supported by NSF grant DMS-1007144.}

\address{Department of Mathematics and Statistics, York University, 4700 Keele Street, North York, Ontario, Canada, M3J 1P3, and Matematicki Institut, Kneza Mihaila 35, Belgrade, Serbia}
\email{ifarah@mathstat.yorku.ca}
\urladdr{http://www.math.yorku.ca/~ifarah}

\address {Department of Mathematics, Statistics, and Computer Science, University of Illinois at Chicago, Science and Engineering Offices M/C 249, 851 S. Morgan St., Chicago, IL, 60607-7045}
\email{isaac@math.uic.edu}
\urladdr{http://www.math.uic.edu/~isaac}

\address{Department of Mathematics and Statistics, McMaster University, 1280 Main Street W., Hamilton, Ontario, Canada L8S 4K1}
\email{hartb@mcmaster.ca}
\urladdr{http://www.math.mcmaster.ca/~bradd}

\begin{abstract}
We observe how a classical model-theoretic fact proves the existence of many strong amalgamation bases for the class of $\R^\omega$-embeddable von Neumann algebras, where $\R$ is the hyperfinite II$_1$ factor.  In particular, we shows that $\R$ itself is a strong amalgamation base.
\end{abstract}

\maketitle

\section{Introduction}

Throughout this note, $\R$ denotes the hyperfinite II$_1$ factor and all von Neumann algebras under consideration are assumed to be tracial.  We say that a separable von Neumann algebra $A$ is \emph{$\R^\omega$-embeddable} if $A$ embeds into some (equivalently all) nonprincipal ultrapowers of $\R$.

We say that an $\R^\omega$-embeddable von Neumann algebra $A$ is an \emph{amalgamation base for the class of $\R^\omega$-embeddable von Neumann algebras} if whenever $A$ is a subalgebra of the $\R^\omega$-embeddable von Neumann algebras $B$ and $C$, then there is an $\R^\omega$-embeddable von Neumann algebra $D$ and embeddings $f:B\to D$ and $g:C\to D$ such that $f|A=g|A$.  If in addition, using the notation from the previous sentence, we can always find $D$, $f$ and $g$ satisfying $f(B)\cap g(C)=f(A)$, we call $A$ a \emph{strong amalgamation base for the class of $\R^\omega$-embeddable von Neumann algebras}.

For the sake of brevity, we shall simply say that $A$ is a (strong) amalgamation base, omitting the phrase ``for the class of $\R^\omega$-embeddable von Neumann algebras.''  We should stress that if we work in the class of all separable von Neumann algebras, then every object is a strong amalgamation base due to the existence of amalgamated free products.  Thus, our convention should be in the back of one's head at all times throughout this note.

The following result appeared in \cite{BDJ}, resulting from a careful analysis of free entropy dimension in amalgamated free products:
\begin{thm} {\cite[Corollary 4.5]{BDJ}}
Suppose that $M_1,M_2$ are $\R^\omega$-embeddable von Neumann algebras and $M:=M_1*_\R M_2$ is the amalgamated free product.  Then $M$ is $\R^\omega$-embeddable.
\end{thm}

In particular, $\R$ is a strong amalgamation base, although the aforementioned result is much more specific, identifying a concrete amalgam, indeed the freest possible amalgam.

In this note, we will show that a large number of $\R^\omega$-embeddable von Neumann algebras are strong amalgamation bases.  This will follow from the classical result of model theory (adapted here to the setting of \emph{continuous model theory}) that any \emph{existentially closed} model of a theory is a strong amalgamation base for the class of models of that theory.  We will then refer to results from \cite{U} which imply that, in the Polish space of separable $\R^\omega$-embeddable von Neumann algebras, the set of existentially closed models is ``large'' in both the topological sense and the measure-theoretic sense.

We should note two nonconstructive features of our observation.  First, we cannot specify exactly what the amalgams look like when considering a strong amalgamation base.  Indeed, the existence of the amalgam follows from the \emph{Compactness Theorem} of (continuous) first-order logic, which is inherently nonconstructive.  (We did attempt to see if one could force our proof below to imply that the amalgamated free product itself is $\R^\omega$-embeddable but were unsuccessful in this endeavor.)  However, this abstract approach has the added benefit of allowing us to deal also with ``nonseparable $\R^\omega$-embeddable von Neumann algebras,'' that is, nonseparable models of the universal theory of $\R$.

Secondly, it might be desirable to know exactly what the existentially closed $\R^\omega$-embeddable von Neumann algebras are.  In this note, we will observe that $\R$ itself is existentially closed, whence recovering the fact that $\R$ is a strong amalgamation base (even allowing nonseparable algebras).  If one were able to axiomatize (in the sense of first-order logic) the existentially closed $\R^\omega$-embeddable von Neumann algebras, then this would have an extremely interesting consequence, namely that the first-order theory of $\R$ is \emph{model-complete}, from which it would follow from \cite[Corollary 3.5]{GHS} that the Connes Embedding Problem (CEP) would have a negative solution.

It is interesting to ask whether or not there is a simple, operator algebraic proof that a large number of $\R^\omega$-embeddable von Neumann algebras are (strong) amalgamation bases.  In relation to the discussion from the previous paragraph, if one could show that \emph{all} $\R^\omega$-embeddable von Neumann algebras are amalagamation bases, then it would follow that the theory of $\R^\omega$-embeddable von Neumann algebras does not have a model companion and that the theory of $\R$ is \emph{not model-complete}, destroying the strategy from \cite{GHS} for providing a negative solution to the CEP.  (Indeed, it is shown in \cite[Proposition 3.2]{GHS} that the only possible model-complete theory of $\R^\omega$-embeddable von Neumann algebras is the theory of $\R$.  If the class of $\R^\omega$-embeddable von Neumann algebras had the amalgamation property, then a model companion would be a model completion, implying that the theory of $\R$ has quantifier elimination, contradicting \cite[Theorem 2.1]{GHS}.)  We discuss this topic in further detail in the last section.

We will work in the setting of continuous model theory.  We refer the reader to \cite{HFS} for a rapid introduction to this setting, where it is also explained how to treat von Neumann algebras as metric structures.  

\section{E.c. models and strong amalgamation bases}

In this section, we let $L$ be a continuous signature and $\k$ a class of $L$-structures.

\begin{df}
We say that $A\in \k$ is an \emph{amalgamation base for $\k$} if whenever $B,C\in \k$ both contain $A$, then there is $D\in \k$ and embeddings $f:B\to D$ and $g:C\to D$ such that $f|A=g|A$.  If, in addition, we can always find $D$, $f$, and $g$ such that $f(B)\cap g(C)=f(A)$, we call $A$ a \emph{strong amalgamation base for $\k$}.
\end{df}

Recall the following definition:

\begin{df}
We say that $A\in \k$ is \emph{existentially closed (e.c.) for $\k$} if for any quantifier-free formula $\varphi(x,y)$, any $a\in A$, and any $B\in \k$ with $A\subseteq B$, we have $$\inf_{c\in A}\varphi^A(c,a)=\inf_{b\in B}\varphi^B(b,a).$$ If $\k$ is the class of models of some theory $T$, we call an e.c. member of $\k$ an \emph{e.c. model of $T$}.
\end{df}

For any $\la$-structure $A$, we let $L(A)$ denote the language $L$ where new constant symbols $c_a$ are added for elements $a\in A$.  We let $D(A)$ denote the atomic diagram of $A$, that is, the set of closed $L(A)$-conditions ``$\sigma(\vec a)=0$,'' where $\sigma(\vec x)$ is a quantifier-free formula, $\vec a$ is a tuple from $A$, and $\sigma^A(\vec a)=0$.  As in classical logic, if $B$ is an $L(A)$-structure that satisfies $D(A)$, then the map sending $a$ to the interpretation of the constant naming $a$ in $B$ is an embedding of $L$-structures.  

We also let $D^+(A)$ denote the set of all closed conditions ``$\sigma(\vec a)\leq \frac{1}{k}$'', where ``$\sigma(\vec a)=0$'' belongs to $D(A)$ and $k\in \n^{>0}$.  Observe that an $L(A)$ structure satisfies $D(A)$ if and only if it satisfies $D^+(A)$.

The following is the continuous logic analog of a classical model-theoretic fact (see \cite[Theorem 3.2.7]{H}, although for some reason there it is assumed that $T$ is $\forall\exists$-axiomatizable, which is surely unnecessary).
\begin{prop}
Suppose that $T$ is an $L$-theory and $A\models T$ is an e.c. model.  Then $A$ is a strong amalgamation base for the models of $T$.
\end{prop}

\begin{proof}
Suppose that $B,C\models T$ both contain $A$.  Without loss of generality, $B\cap C=A$.  For $c\in C\setminus A$, set $\delta_c:=d(c,A)>0$.  It suffices to show that the following set of $L(BC)$-conditions is satisfiable:
$$T\cup D^+(B)\cup D(C)\cup \{d(b,c)\geq \delta_c \ | \ b\in B\setminus A, c\in C\setminus A\}.$$  Suppose that this is not the case.  Then there is $k\in \n^{>0}$, $\vec b=(b_1,\ldots,b_n)$ from $B\setminus A$, a quantifier-free formula $\chi(\vec b,\vec d)$, where $\vec d\in A$ and $\chi^B(\vec b,\vec d)=0$, and $c_1,\ldots,c_n$ from $C\setminus A$ such that $$T\cup \{\chi(\vec b,\vec d)\leq \frac{1}{k}\}\cup D(C)\cup \{d(b_i,c_i)\geq \delta_{c_i} \ | \ i=1,\ldots,n\}$$ is unsatisfiable.  Consequently, the set of $L(C)$-conditions
$$T\cup \{\chi(\vec x,\vec d)\leq \frac{1}{k}\}\cup D(C)\cup \{d(x_i,c_i)\geq \delta_{c_i} \ | \ i=1,\ldots,n\}$$ is unsatisfiable.  Since $A$ is e.c., there are $\vec a\in A$ such that $\chi^A(\vec a,\vec d)\leq \frac{1}{k}$, whence $\chi^C(\vec a,\vec d)\leq \frac{1}{k}$.  Consequently, there is $i\in \{1,\ldots,n\}$ such that $d(a_i,c_i)<\delta_{c_i}$, a contradiction.
\end{proof}

Observe in the previous proof that we could have replaced $D(C)$ by the full elementary diagram of $C$, whence we can always assume that the amalgam is an elementary extension of $C$.  

\

\section{Application to $\R^\omega$-embeddable von Neumann algebras}

We  apply the results of the previous section to the case that $L$ is the language used for studying tracial von Neumann algebras and $T$ is the universal theory of $\R$:  $T=\Th_\forall(\R)$.  The relevant observation is that a separable model of $T$ is just an $\R^\omega$-embeddable  von Neumann algebra.  We thus see that:

\begin{cor}
Any e.c. $\R^\omega$-embeddable  von Neumann algebra is a strong amalgamation base for the class of $\R^\omega$-embeddable von Neumann algebras.  
\end{cor}

It is worth pointing out that an e.c. $\R^\omega$-embeddable von Neumann algebra is necessarily a II$_1$ factor; see \cite{GHS}.

The previous corollary begs the question:  what are the e.c. $\R^\omega$-embeddable II$_1$ factors?  We can identify one:

\begin{prop}
$\R$ is an e.c. model of $\Th_\forall(\R)$.
\end{prop}

\begin{proof}
Suppose $\varphi(x,y)$ is quantifier-free, $a\in \R$ and $\R\subseteq M$, where $M$ is a model of $\operatorname{Th}_\forall(\R)$.  Without loss of generality (using Downward L\"owenheim-Skolem), we may assume that $M$ is separable.  Then we have an embedding $f:M \to\R^\u$, where $\R^\u$ is some nonprincipal ultrapower of $\R$.  Note then that $$(\inf_x\varphi(x,a))^\R\geq \inf(\varphi(x,a))^M\geq (\inf_x(\varphi(x,f(a)))^{\R^\u}.$$  Since $f|\R$ is elementary (see \cite[Lemma 3.1]{GHS}), the ends of the inequality are equal, whence $(\inf_x\varphi(x,a))^\R=(\inf_x\varphi(x,a))^M$.
\end{proof}

We should point out that, in the previous proof, we relied on the fact that every embedding of $\R$ into $\R^\u$ is elementary, which is merely a consequence of the fact that every embedding of $\R$ into $\R^\u$ is unitarily conjugate to the diagonal embedding (a result perhaps known even to von Neumann).
\begin{cor}
$\R$ is a strong amalgamation base for the class of $\R^\omega$-embeddable tracial von Neumann algebras.
\end{cor}
\newpage
We recall the following:

\begin{facts}

\

\begin{itemize}
\item Any $\R^\omega$-embeddable von Neumann algebra embeds into an e.c. $\R^\omega$-embeddable von Neumann algebra.  (This follows from the general model-theoretic fact that any infinite model of an $\forall\exists$-axiomatizable theory in a countable language embeds into an e.c. model of that theory of the same cardinality.)
\item There is a family $(M_\alpha)_{\alpha<2^{\aleph_0}}$ of $\R^\omega$-embeddable II$_1$ factors such that, for any II$_1$ factor $M$, at most countably many of the $M_\alpha$'s embed into $M$ (see \cite{NPS}).
\end{itemize}
\end{facts}

Consequently, we have:

\begin{cor}
There are $2^{\aleph_0}$ many strong amalgamation bases for the class of $\R^\omega$-embeddable von Neumann algebras.
\end{cor}

Let us end by briefly explaining what we mean when we say that a ``large'' number of $\R^\omega$-embeddable von Neumann algebras are e.c.  Throughout this paragraph, we freely use results from \cite{U}, although specialized to the special case at hand.  Let $\mathcal G$ denote the space of all $\R^\omega$-embeddable von Neumann algebras $M$ equipped with a distinguished countable dense subset $M_0\subseteq M$, enumerated as $(m_i:i<\omega)$.  We can define a topology on $\mathcal G$ by declaring sets of the form $$\{M \in \mathcal G \ : \ \varphi^M(m_{i_1},\ldots,m_{i_n})<\epsilon\}$$ to be basic open sets, where $\varphi$ is a quantifier-free formula (e.g. the trace of a $*$-polynomial), $i_1,\ldots,i_n\in \n$ and $\epsilon\in \r^{>0}\cup \{+\infty\}$.  In this way, $\mathcal G$ becomes a Polish space.  It is a consequence of results from \cite{U} that the set of e.c. elements of $\mathcal G$ is dense in $\mathcal G$ and any reasonable probability measure on $\mathcal G$ gives the set of e.c. elements full measure.

\

\section{CEP and Model-completeness revisited}

As mentioned in the Introduction, if one could show that all $\R^\omega$-embeddable von Neumann algebras are amalgamation bases, then it would follow that $\Th(\R)$ is not model-complete.  We now show that it suffices to show that a particular $\R^\omega$-embeddable von Neumann algebra is an amalgamation base in order to conclude that $\Th(\R)$ is not model-complete.

The idea stems from the proof given in \cite[Section 2]{GHS} that $\Th(\R)$ does not have quantifier elimination.  There, we considered a finitely generated, $\R^\omega$-embeddable algebra $M$, an embedding $\pi:M\to \R^\u$, and an automorphism $\alpha$ of $M$ for which there did not exist a unitary $u\in \R^\u$ satisfying $\pi(\alpha(x))=u\pi(x)u^*$ for all $x\in M$.  In order to show that quantifier elimination failed, it sufficed to observe that, setting $N:=M\rtimes_\alpha \Z$, one cannot extend $\pi$ to an embedding of $N$ into $\R^\omega$ (after identifying $M$ with its image under the natural inclusion $i:M\to N$).

Keeping with the notation of the previous paragraph, let $\R'$ denote a separable elementary substructure of $\R^\u$ containing $\pi(M)$ and suppose that there exists an $\R^\omega$-embeddable von Neumann algebra $N'$ and embeddings $f:N\to N'$ and $g:\R'\to N'$ such that $f\circ i=g\circ \pi$ (e.g. if $M$ were an amalgamation base).  Fix an embedding $h:N'\to \R^\u$.  Let $u'\in N$ denote the unitary associated to the generator of $\mathbb Z$.  \emph{If} $\Th(\R)$ \emph{were} model-complete, then $h\circ g$ would be an elementary embedding.  Since $M$ is finitely generated, it would follow that there does not exist a unitary $u\in \R^\u$ such that $h(g(\pi(\alpha(x)))=uh(g(\pi(x))u^*$ for all $x\in M$, contradicting the fact that $u:=h(f(u'))$ is such a unitary.

We summarize this as follows:

\begin{prop}
Suppose that there exists a finitely generated $\R^\omega$-embeddable von Neumann algebra $M$ with the following two properties:
\begin{itemize}
\item There is an embedding $\pi:M\to \R^\u$ and an automorphism $\alpha$ of $M$ such that there does \emph{not} exist a unitary $u$ of $\R^\u$ satisfying $\pi(\alpha(x))=u\pi(x)u^*$ for all $x\in M$.
\item $M$ is an amalgamation base.
\end{itemize}
Then $\Th(\R)$ is \emph{not} model-complete.
\end{prop}

\end{document}